\documentclass[a4paper,11pt]{amsart}

\usepackage[latin1]{inputenc}

\usepackage{latexsym}

\usepackage{amssymb}
\usepackage{amsfonts}
\usepackage{amsmath}
\usepackage{amsthm}
\usepackage{textcomp}

\usepackage{comment}

\usepackage[dvips]{color}
\usepackage{graphicx}

\newtheorem{theorem}{Theorem}[section]
\newtheorem{lemma}[theorem]{Lemma}

\newtheorem{proposition}[theorem]{Proposition}
\newtheorem{remark}[theorem]{Remark}

\theoremstyle{definition}
\newtheorem{definition}[theorem]{Definition}

\def\C{{\mathbb C}}

\title[Carleson measures for Hardy and Bergman spaces]{Carleson measures for Hardy and Bergman spaces in the quaternionic unit ball}

\author[I. Sabadini]{Irene Sabadini}
\address{Irene Sabadini\\ Dipartimento di Matematica\\ Politecnico di Milano\\ Via Bonardi 9\\ I-20133 Milano\\ Italy}
\email{irene.sabadini@polimi.it}

\author[A. Saracco]{Alberto Saracco}
\address{Alberto Saracco\\ Dipartimento di Matematica e Informatica\\ Universit\`a di Parma\\ Parco Area delle Scienze 53/A\\ I-43124 Parma\\ Italy}
\email{alberto.saracco@unipr.it}

\date{\today}

\thanks{This paper was partially supported by FIRB 2012 Differential Geometry and Geometric Function Theory 
(financed by Italian MIUR)}

\begin{document}

\begin{abstract}
We study a characterization of slice Carleson measures and of Carleson measures for both the Hardy spaces $H^p(\mathbb B)$ and the Bergman spaces $\mathcal A^p(\mathbb B)$ of the quaternionic unit ball $\mathbb B$. In the case of Bergman spaces, the characterization is done in terms of the axially symmetric completion of a pseudohyperbolic disc in a complex plane. We also show that a characterization in terms of pseudohyperbolic balls is not possible.
\end{abstract}

\maketitle

\section{Introduction}
Carleson measures have been introduced around 1960 by Carleson, see \cite{carleson}, to prove the corona theorem and to solve interpolation problems. A finite, positive, Borel measure $\mu$ on the open unit disc $\mathbb D\subset\mathbb C$ is called a Carleson measure for the Hardy space $H^p(\mathbb D)$ if
$$
\int_{\mathbb D} |f(z)| d\mu \leq C \| f\|_{H^p}, \qquad f\in H^p(\mathbb D)
$$
where the constant $C$ depends only on $p\in(0,\infty)$. Carleson also proved that the measure $\mu$ is Carleson for $H^p(\mathbb D)$ if and only if for all Carleson squares of sidelenght $1-r$, $r\in (0,1)$, $\theta_0\in[0, 2\pi]$
$$
S(\theta_0, r)=\{\rho e^{i\theta}\in \mathbb D\ \ : \ \ r\leq \rho < 1, \  |\theta-\theta_0 |\leq 1-r\}
$$
the condition
$$
\mu (S(\theta_0, r))\leq C (1-r)
$$
holds for some constant $C>0$. This condition shows that,  for a measure, the property of being Carleson for $H^p(\mathbb D)$ does not depend on $p$.

Later, Hastings \cite{H} (but see also \cite{O} for a wider context) proved a similar characterization for the measures which are Carleson for the Bergman space $\mathcal{A}^p(\mathbb D)$ showing that a finite, positive, Borel measure $\mu$ is a Carleson measure for $\mathcal{A}^p(\mathbb D)$ if and only if
$$
\mu (S(\theta_0, r))\leq C (1-r)^2
$$
for some constant $C>0$ for all Carleson squares $S(\theta_0, r)$ of sidelenght $1-r$. In particular,  for a measure, the property of being Carleson for $\mathcal{A}^p(\mathbb D)$ does not depend on $p$.
\\
The Carleson boxes are clearly not invariant under automorphism of the unit disc, so in order to obtain a characterization of the Carleson measures in an invariant way, one should consider the hyperbolic geometry of the disc and consider instead the pseudohyperbolic discs, see the work by Luecking \cite{luecking} and also \cite{DS}.
A pseudohyperbolic disc $\Delta (z_0, r)$ with center $z_0$ and radius $r>0$ is defined as
$$
\Delta (z_0, r)=\left\{z\in\mathbb C\ \ : \ \ \left|\frac{z-z_0}{1-\bar{z}_0 z}\right|<r\, \right\}.
$$
Luecking showed that a finite, positive measure is Carleson for $\mathcal{A}^p(\mathbb D)$ if and only if
$$
\mu(\Delta (z_0, r)) \leq C |\Delta (z_0, r)|
$$
for some constant $C>0$ depending only on $r$ and for pseudohyperbolic discs $\Delta (z_0, r)$, $z_0\in\mathbb D$, for some (and hence for all) $r\in (0,1)$.
\\
The characterization of Carleson measures of the Bergman spaces was then extended to more general settings, like strongly pseudoconvex domains, see e.g. \cite{CM} and \cite{AbS}.
\\
In this paper we will study, in the quaternionic setting, the characterization of Carleson measures in the Hardy and Bergman spaces on the open unit ball $\mathbb B$. This study was initiated by Arcozzi and Sarfatti in \cite{AS} who treated the case of the Hardy space $H^2(\mathbb B)$. Here we extend their result to the case of $H^p(\mathbb B)$ and to the case of the Bergman spaces $\mathcal{A}^p(\mathbb B)$,  for any $p$. It is important to note that, in the complex case, to have a characterization of Carleson measures in $H^2(\mathbb D)$ allows to prove the result for $H^p(\mathbb D)$ for any $p$. In this setting, the method that applies in the complex case is not so immediately applicable as it is not true, in general, that the $p/2$-power of a slice regular function is slice regular.

Furthermore, it is a delicate question to understand which is an appropriate notion of pseudohyperbolic metric and pseudohyperbolic ball. In fact, the function $\rho:\ \mathbb B\times \mathbb B\to \mathbb R^+$ defined by
\[
\rho(q,\alpha)=\left|\frac{\alpha-q}{1-q\bar\alpha}\right|
\]
is a distance only if $q$ belongs to the complex plane $\mathbb C_I$ containing $\alpha$.
We can then define the disc $\Delta_I (\alpha, r)\subset \mathbb C_I$ with center $\alpha$ and radius $r\in (0,1)$ and its axially symmetric completion $\Delta (\alpha, r)$.
The distance $\rho$ (later on denoted by $\rho_I$) is used to show a result which characterizes Carleson measure $\mu$ in the Bergman spaces $\mathcal{A}^p(\mathbb B)$ as the measures for which
$$
\mu(\Delta (\alpha, r))\leq C |\Delta_I (\alpha, r)|
$$
for some (and hence for all) $r\in(0,1)$, where $C$ is a constant depending on $r$ only.
More in general, when $q\in\mathbb B$, the function $\rho$ should be substituted by its slice regular extension, namely by the function, still denoted by $\rho$,
$$
\rho(q,\alpha)=\left|(1-q\bar\alpha)^{-*} * (q-\alpha)\right|.
$$
The function $\rho(q,\alpha)$ is now a distance in $\mathbb B$, see \cite{AS1}, which allows to define hyperbolic balls in $\mathbb H$. However, we show that it is not possible to characterize Carleson measures in terms of these balls. This is a major difference with respect to the complex case  and shows that the pseudohyperbolic metric does not carry the intrinsic information needed to characterize Carleson measures.\\
The plan of the paper is as follows. After the introduction and some preliminary results, in Section 3 we characterize slice Carleson measures and then Carleson measures in $H^p(\mathbb B)$ for any $p$. In Section 4 we prove the characterization in the case of the Bergman spaces $\mathcal{A}^p(\mathbb B)$ for any $p$ in terms of the axially symmetric completion of a pseudohyperbolic disc in a complex plane. In Section 5 we study the quaternionic analog of pseudohyperbolic balls, also showing that it is not possible to characterize Carleson measures for the Bergman spaces in terms of these balls.
\\
{\bf Acknowledgments}. The authors thank Giulia Sarfatti for useful discussions.
\section{Preliminary results}
There are several ways to generalize the notion of holomorphy to the quaternionic setting. In the past few years, slice regularity attracted the attention of several researchers and this will be the notion of holomorphy we will consider in this paper. We repeat here some useful definitions and results on this function theory, see \cite{book_functional}, \cite{bookgss}.\\
First of all, we recall that the skew field of quaternions $\mathbb H$ is defined as the set of elements $q=x_0+ x_1 i+x_2 j+x_3k$ where $x_\ell\in\mathbb R$, $\ell=0,\ldots ,3$ and $i^2=j^2=-1$, $ij=-ji=k$. For any $q\in\mathbb H$, its conjugate $\bar q$ is defined as $\bar q=x_0- x_1 i-x_2 j-x_3k$, moreover the norm of $q$
is the Euclidean norm $|q|=(x_0^2+x_1^2+x_2^2+x_3^2)^{1/2}$. It is immediate that $|q|^2=q\bar q=\bar q q$.
\\
By $\mathbb{S}$ we denotes the $2$-dimensional unit sphere of purely imaginary quaternions, namely
$$\mathbb{S}=\{q = i x_{1} + j x_{2} + k x_{3}, \mbox{ such that } x_{1}^{2}+x_{2}^{2}+x_{3}^{3}=1\}.$$
An element $I\in \mathbb{S}$ is such that $I^{2}=-1$ and thus the elements of $\mathbb{S}$ are also
called imaginary units. For any fixed $I\in\mathbb{S}$ the set $\mathbb{C}_I:=\{x+Iy; \ |\ x,y\in\mathbb{R}\}$ is a complex plane, moreover $\mathbb{H}=\bigcup_{I\in\mathbb{S}} \mathbb{C}_I$. Obviously, the real axis belongs to $\mathbb{C}_I$ for every $I\in\mathbb{S}$.\\
Any non real quaternion $q$ is uniquely associated to an element in $\mathbb{S}$, specifically we can set $I_q:=( i x_{1} + j x_{2} + k x_{3})/|  i x_{1} + j x_{2} + k x_{3}|$. It is obvious that $q$ belongs to the complex plane $\mathbb{C}_{I_q}$. Thus the quaternion $q=x_0+ x_1 i+x_2 j+x_3k$ can be written as $q =x_0+ I_q y_0$ where $y_0=
|  i x_{1} + j x_{2} + k x_{3}|$. The set of elements of the form $x_0+Iy_0$ when $I$ varies in $\mathbb S$ will be denoted by $[q]$; it is a $2$-dimensional sphere consisting of elements with the same real part and the same modulus as $q$. An open set $U$ is said to be axially symmetric if for any $q\in U$ the whole sphere $[q]$ is contained in $U$.
\\

\begin{definition}
 Let $U$ be an open set in $\mathbb{H}$ and $f:\, U\to \mathbb{H}$ be real differentiable. The function $f$ is said to be (left) slice regular  or (left) slice hyperholomorphic
if for every $I\in \mathbb{S}$, its restriction $f_{I}$ to the complex plane ${\mathbb{C}}_{I}=\mathbb{R}+ I \mathbb{R}$ passing through origin
and containing $I$ and $1$ satisfies
$$\overline{\partial}_{I}f(x+I y):=\frac{1}{2}\left (\frac{\partial}{\partial x}+I \frac{\partial}{\partial y}\right )f_{I}(x+I y)=0,$$
on $U\cap \mathbb{C}_{I}$.
Analogously, a function is said to be right slice regular in $U$ if
$$(f_{I}{\overline{\partial}}_{I})(x+I y):=\frac{1}{2}\left (\frac{\partial}{\partial x}f_{I}(x +I y)+\frac{\partial}{\partial y}f_{I}(x+I y) I\right )=0,$$
on $U\cap \mathbb{C}_{I}$.
\end{definition}
In the sequel, given an open set $U\subset\mathbb H$ and $I\in\mathbb S$, we will sometimes write $U_I$ to denote $U\cap\mathbb C_I$.\\
Let us consider
 $I,J\in\mathbb{S}$ with $I$ and $J$ orthogonal, so that $I,J,IJ$ is an orthogonal
basis of $\mathbb H$. Let us write the restriction  $f_I(x+Iy)=f(x+Iy)$ of $f$ to the complex plane $\mathbb C_I$  as
$$
f=f_0+If_1+Jf_2+Kf_3=F+GJ
$$
where $f_0+If_1=F$, and $f_2+If_3=G$.
This observation immediately gives the following result:
\begin{lemma}[Splitting Lemma]
If $f$ is a slice regular function
on $U$, then for every $I \in \mathbb{S}$, and every $J\in\mathbb{S}$,
orthogonal to $I$, there are two holomorphic functions
$F,G:U\cap \mathbb{C}_I \to \mathbb{C}_I$ such that for any $z=x+Iy$
$$f_I(z)=F(z)+G(z)J.$$
\end{lemma}
Another useful result is the following:
\begin{theorem}
 A function $f:\ {B}(0;{r}) \to \mathbb{H}$ is slice regular on ${B}(0;{r})$ if and only if it has a series representation of the form
\begin{equation}\label{powerseries}
f(q)=\sum_{n=0}^{\infty}q^{n}\frac{1}{n !}\cdot \frac{\partial^{n} f}{\partial x^{n}}(0)=\sum_{n=0}^{\infty}q^{n} a_n,
\end{equation}
uniformly convergent on ${B}(0;{r})$.
\end{theorem}
To state the next fundamental property of slice regular functions, we recall that a domain $U$ in $\mathbb H$ is called s-domain if $U\cap\mathbb C_I$ is connected for all $I\in\mathbb S$.
\begin{theorem}[Representation Formula]\label{Repr_formula} Let
$f$ be a slice regular function defined on an axially symmetric s-domain $U\subseteq  \mathbb{H}$. Let
$J\in \mathbb{S}$ and let $x\pm Jy\in U\cap\mathbb C_J$.  Then the following equality holds for all $q=x+Iy \in U$:
\begin{equation}\label{distribution_mon}
\begin{split}
f(x+Iy) &=\frac{1}{2}\Big[   f(x+Jy)+f(x-Jy)\Big]
+I\frac{1}{2}\Big[ J[f(x-Jy)-f(x+Jy)]\Big]\\
&= \frac{1}{2}(1-IJ) f(x+Jy)+\frac{1}{2}(1+IJ) f(x-Jy).
\end{split}
\end{equation}
\end{theorem}
This formula is important as it allows to reconstruct the values of a slice regular function when its values are known on a complex plane. In particular, we have
that if $J\in\mathbb S$, $U$ is an axially symmetric set, and $f:U\cap \mathbb C_J\to \mathbb H$ satisfies $\overline{\partial}_J f=0$ then the function
\begin{equation}\label{ext}
{\rm ext}(f)(x+Iy)=\frac{1}{2}\Big[   f(x+Jy)+f(x-Jy)\Big]
+I\frac{1}{2}\Big[ J[f(x-Jy)-f(x+Jy)]\Big]
\end{equation}
is the unique slice regular extension of $f$ to $U$.
\\
Given an open set $U_J\subseteq\mathbb C_J$ we can define the so-called axially symmetric completion as
$$
U=\bigcup_{I\in\mathbb S} \{x+Iy\ :\ x+Jy\in U_J\}.
$$
It is clear that the sum of two slice regular functions is slice regular. However, in general, the pointwise product of two slice regular functions is not slice regular. Bearing in mind the definition of product between two power series
$f(q)=\sum_{n=0}^\infty q^n a_n$, $g(q)=\sum_{n=0}^\infty q^n b_n$
with coefficients in a ring converging in $B(0;R)$ for some $R>0$ (in particular two polynomials) we define
$$
(f * g)(q)=\sum q^n c_n, \qquad c_n=\sum_{k=0}^n a_k b_{n-k}.
$$
This notion can be extended to functions $f,g$ slice regular on an axially symmetric s-domain, but we do not enter the details here and we refer the interested reader to \cite{book_functional,bookgss}. A notion more important for us will be the one of inverse of a function with respect to he $*$-multiplication. To this end, we have to introduce some more notations which are of independent interest limiting the definitions to the case of functions slice regular on a ball $B(0;R)$:
\begin{definition} Given the  function $f(q)=\sum_{n=0}^\infty q^n a_n$ slice regular on $B(0;R)$, we define its regular conjugate
$$
f^c(q)=\sum_{n=0}^\infty q^n \overline{a_n}
$$
and its symmetrization (or normal form)
$$
f^s(q)=(f*f^c)(q)=(f^c*f)(q).
$$
The inverse of $f$ with respect to the slice regular multiplication is denoted by $f^{-*}$ and is given by
$$
f^{-*}(q)=(f^s(q))^{-1} f^c(q).
$$
It is defined for $q\in B(0;R)$ such that $f^s(q)\not=0$.
\end{definition}
It is not true, in general, that the composition of two slice regular functions, when it is defined, is a slice regular function. However, it is true in a suitable subclass of slice regular functions defined below:
\begin{definition}
A function $f$  slice regular in an open set $U$ is called quaternionic intrinsic if
$$
f(U\cap \mathbb{C}_I)\subseteq  \mathbb{C}_I,\ \  \forall I\in \mathbb{S}.
$$
\end{definition}
If $U$ is an axially symmetric open set, it can be shown that a slice regular function $f$ is quaternionic intrinsic if and only if $f(q)=\overline{f(\bar q)}$ for all $q\in U$. This also justifies the terminology "intrinsic" that comes from the analogous property in the complex case. It is useful to note that a function slice regular on the ball  $B(0;R)$ with center at the origin and radius $R$ is quaternionic intrinsic if and only if its
power series expansion has real coefficients.
Moreover, we have
\begin{proposition}\label{comp}
Let $f$, $g$ be slice regular in the open sets $U'$, $U\subseteq\mathbb H$, respectively, $g(U)\subseteq U'$ and let $g$ be quaternionic intrinsic.
Then
$f(g(q))$ is slice regular in $U$.
\end{proposition}
\section{Carleson measures in Hardy spaces}

Purpose of this section is to define the (slice) Carleson measures for $H^p(\mathbb B)$ and to prove their characterization. Let us first recall the definition of Hardy spaces $H^p(\mathbb B)$, see \cite{dGS}:
\begin{definition}
Let $p\in (0,+\infty)$. We define
$$
H^p(\mathbb B)=\{f:\, \mathbb B\to\mathbb H\ | \ f\ {\rm is\ slice\ regular\ and}\ \|f\|_p<+\infty\},
$$
where
$$
\| f\|_p =\sup_{I\in \mathbb S}\lim_{r\to 1^-} \left( \int_0^{2\pi} |f(re^{I\theta})|^p d\theta\right)^{1/p}.
$$
\end{definition}
In the sequel, we will also be in need of the following spaces:
\begin{definition}
Let $p\in (0,+\infty)$ and $\mathbb B_I=\mathbb B\cap\mathbb C_I$. We define
$$
H^p(\mathbb B_I)=\{f:\, \mathbb B\to\mathbb H\ | \ f\ {\rm is\ slice\ regular\ and}\ \|f_I\|_p<+\infty\},
$$
where
$$
\| f_I\|_p =\lim_{r\to 1^-} \left( \int_0^{2\pi} |f(re^{I\theta})|^p d\theta\right)^{1/p}.
$$
\end{definition}
\begin{remark}\label{Hpall}{\rm
It is clear that $H^p(\mathbb B)\subseteq H^p(\mathbb B_I)$ for any $p\in (0,+\infty)$. However, it can be proved that $f\in H^p(\mathbb B)$ if and only if $f\in H^p(\mathbb B_I)$ for some (and hence for all) $I\in\mathbb S$; this is a quite general phenomenon which, basically, follows from the Representation Formula, see \cite{CCGG,CGS4}, and \cite{dGS} for this specific case.}
\end{remark}

In the sequel we will also use the reproducing kernel for $H^2(\mathbb B)$ that was introduced in \cite{acs2}:
\begin{definition} The reproducing kernel of $H^2(\mathbb B)$ is the function
\begin{equation}\label{kernel}
k(q,w)=(1-2{\rm Re}(w) q+|w|^2q^2)^{-1}(1-qw)=(1-\bar q \bar w)
(1-2{\rm Re}(q) \bar w+|q|^2\bar w^2)^{-1}.
\end{equation}
\end{definition}
The kernel $k(q,w)$ is defined for all $q$ such that $1-2{\rm Re}(w) q+|w|^2q^2\not=0$ (or, equivalently, for all $w$ such that $1-2{\rm Re}(q) \bar w+|q|^2\bar w^2$).
It is the sum of the series $\sum_{n=0}^{+\infty} q^n\bar w^n$.
 Moreover:
\begin{itemize}
\item[a)] $k(q,w)$ is slice regular in $q$ and right slice regular in $\bar w$;
\item[b)] $\overline{k(q,w)}=k(w,q)$.
\end{itemize}
The function $k(q,w)$ is such that $k(q,w)*(1-q\bar w)=(1-q\bar w)*k(q,w)=1$, as it can be easily verified
and thus we can write $k(q,w)=(1-q\bar w)^{-*}$ where the $*$-inverse is computed with respect to the variable $q$
(but note that $k(q,w)$ also equals the right $*$-inverse of $(1-q\bar w)$ in the variable $\bar w$).

\begin{definition}
A finite, positive, Borel measure $\mu$ on the unit ball $\mathbb B\subset\mathbb H$ is said to be
 a {\it Carleson measure for $H^p(\mathbb B)$} if for any $f\in H^p(\mathbb B)$
\begin{equation} \int_{\mathbb B}|f(q)|^p\,d\mu(q)\leq C \parallel f\parallel^p_{H^p(\mathbb B)}\lesssim \parallel f\parallel^p_{H^p(\mathbb B)}\,,
\end{equation}
the constant $C$ in the estimate depending only on $\mu$.
\end{definition}
Any measure $\mu$ on $\mathbb B$ can be decomposed as $\mu=\mu_{\mathbb R}+\tilde\mu$ where ${\rm supp}(\mu_{\mathbb R})\subseteq\mathbb B\cap\mathbb R$ and $\tilde\mu (\mathbb B\cap\mathbb R)=0$. Obviously, a measure $\mu$ is a Carleson measure if and only if $\mu_{\mathbb R}$ and $\tilde\mu$ are Carleson.\\
The Disintegration Theorem implies that any finite measure $\mu$ on $\mathbb B$ such that $\mu (\mathbb B\cap\mathbb R)=0$ can be decomposed as
$$
\mu (x+Iy)= \mu^+_I(x+Iy)\, \nu( I)
$$
where $\mu^+_I$ is a probability measure on
$$
\mathbb B_I^+=\{x+Iy\in\mathbb B\ :\ y \geq 0\}
$$
and $\nu$ is the measure on the Borel sets $E$ of the sphere $\mathbb S$ defined by
$$
\nu(E)=\mu\{x+Iy\in\mathbb B\ :\ y>0,\  I\in E\}.
$$
With this notation, for any $f\in L^1(\mathbb B, d\mu)$ we have (see \cite{AS}, \cite{GGJ}):
\begin{equation}\label{misuremunu}
\int_{\mathbb B} f(x+Iy)\, d\mu(x+Iy) = \int_{\mathbb S} \int_{\mathbb B_I^+} f(x+Iy)\, d\mu_I^+(x+Iy)\, d\nu(I).
\end{equation}
Thus, if $\mu$ is a finite measure on $\mathbb B$ written in the above form $\mu=\mu_{\mathbb R}+\tilde\mu$,
applying the Disintegration Theorem to $\tilde\mu$ we can write
\begin{equation}\label{misure}
\int_{\mathbb B}f(x+Iy)\,d\mu(x+yI)\, =\, \int_{\mathbb B\cap\mathbb R} f(x)\,d\mu_{\mathbb R}(x) +\int_{\mathbb S}\int_{\mathbb B_I^+}f(x+yI)\,d\tilde{\mu}_I^+(x+yI)\, d\nu(I).
\end{equation}

We can introduce, for any $I\in\mathbb S$, the measure $\mu_I=\mu_{\mathbb R}+\tilde\mu_I^++\tilde\mu_{-I}^+$ which is the restriction of the measure $\mu$ to $\mathbb B_I$, and
we can write the definition of slice Carleson measure:
\begin{definition}
  A finite, positive, Borel measure $\mu$ is said to be a {\it slice Carleson measure for $H^p(\mathbb B)$} if for any $f\in H^p(\mathbb B)$ and any $I\in\mathbb S$
\begin{equation}\label{9}
\int_{\mathbb B_I}|f(x+yI)|^p\,d\mu_I(x+yI)\, \leq C \parallel f\parallel^p_{H^p(\mathbb B)}\,\lesssim\, \parallel f\parallel^p_{H^p(\mathbb B)}\,,
\end{equation}
the constant $C$ in the estimate depending only on $\mu$.
\end{definition}
We can rewrite the left hand side of \eqref{9} obtaining the condition:
\begin{equation}\label{eq9}
\begin{split}
\int_{\mathbb B_I}|f(x+yI)|^p\,d\mu_I(x+yI)\, &=\, \int_{\mathbb B\cap\mathbb R} |f(x)|^p\,d\mu_{\mathbb R}(x)+\int_{\mathbb B_I^+}|f(x+yI)|^p\,d\tilde{\mu}_I^+(x+yI)\\
&+\int_{\mathbb B_{-I}^+}|f(x+y(-I))|^p\,d\tilde{\mu}_{-I}^+(x+y(-I))\\
&\lesssim\, \parallel f\parallel^p_{H^p(\mathbb B)}\, .
\end{split}
\end{equation}
 We now define the analog of the Carleson box (or Carleson square, see \cite{DS}) in this framework, see \cite{AS}.
\begin{definition}
Let $q=re^{J\theta_0}$ be an element in $\mathbb B$ and let $A_I(\theta_0,r)$ be the arc of $\partial\mathbb B_I$ defined by
\[
A_I(\theta_0,r)=\{e^{I\theta}\in\partial\mathbb B_I\ :\ |\theta -\theta_0|\leq 1-r\},
\]
and let $S_I(\theta_0,r)$ be the Carleson box in the plane $\mathbb C_I$ defined by
\[
S_I(\theta_0,r)=\{\rho e^{I\theta}\in\mathbb B_I\ :\ e^{I\theta}\in A_I(\theta_0,r), \, r \leq \rho < 1\}.
\]
We say that the set
\[
S(\theta_0,r)=\underset{I\in\mathbb S}{\bigcup} S_I(\theta_0,r)
\]
is a symmetric box.
\end{definition}
\begin{remark}{\rm
For any $q=re^{I\theta_0}$, the length of the arc $A_I(\theta_0,r)$, denoted by $|A_I(\theta_0,r)|$, is $2(1-r)$ and it is independent of $I\in\mathbb S$.}
\end{remark}
The relation between slice Carleson measures and Carleson measures for $H^2(\mathbb B)$ given in Proposition 3.1 in \cite{AS} immediately extends to the spaces $H^p(\mathbb B)$:
\begin{proposition}\label{sliceCarleson} If a measure $\mu$ is slice Carleson for the Hardy space $H^p(\mathbb B)$, then it is also Carleson.
\end{proposition}
\begin{proof} Let $\mu$ be a slice Carleson measure for $H^p(\mathbb B)$, so that it is also finite by assumption. Moreover, using \eqref{misure}, we obtain that for any $f\in H^p(\mathbb B)$,
\begin{eqnarray*}
\int_{\mathbb B}|f(q)|^p\,d\mu(q)\ &=& \ \int_{\mathbb B\cap\mathbb R} |f(x)|^p\,d\mu_{\mathbb R}(x)\, +\, \int_{\mathbb S}\,d\nu(I)\, \int_{\mathbb B_I^+} |f(z)|^p\,d\tilde{\mu}^+_I(z)\\
&\lesssim&\ \parallel f\parallel^p_{H^p(\mathbb B)}\, +\, \int_{\mathbb S}\,d\nu(I)\, \parallel f\parallel^p_{H^p(\mathbb B)}\ \lesssim \ \parallel f\parallel^p_{H^p(\mathbb B)}.
\end{eqnarray*}
\end{proof}
Next result is the generalization of Proposition 3.3 in \cite{AS} to the present setting.
\begin{proposition}\label{sC} A finite, positive, Borel measure $\mu$ is a slice Carleson measure for $H^p(\mathbb B)$ if and only if for all $I\in\mathbb S$, and for all $z=re^{I\theta_0}\in\mathbb B_I$
\begin{equation}\label{sCeq}
\mu_I(S_I(r,\theta_0))\ \lesssim \ |A_I(r,\theta_0)|.
\end{equation}
\end{proposition}
\begin{proof}
 Assume that $\mu$ is slice Carleson. Then, in view of Remark \ref{Hpall}, (\ref{9}) holds for all $f\in H^p(\mathbb B_I)$, then it holds, in particular, also for any $F\in  H^p(\mathbb B_I)$ that maps $\mathbb B_I$ into itself. This means that $\mu_I$ is a Carleson measure for complex Hardy space $H^p(\mathbb D)\subset H^p(\mathbb B_I)$. The classical characterization theorem for Carleson measures of complex Hardy spaces (-from now on {\it classical Carleson theorem} - see e.g. \cite{D}) gives us the statement.

 Conversely, let $\mu$ be such that (\ref{sCeq}) holds. Then, for the classical Carleson theorem, $\mu_I$ is a Carleson measure for $H^p(\mathbb D)\subset H^p(\mathbb B_I)$. According to the Splitting Lemma, any $f\in H^p(\mathbb B)$ restricted to $\mathbb B_I$ decomposes as $f(z)=F(z)+G(z)J$, with $J\in\mathbb S$, $J\perp I$, and $F$, $G$ holomorphic on $\mathbb B_I$. Thus
\begin{eqnarray*}
\int_{\mathbb B_I}|f(z)|^p\,d\mu_I(z)\, &=& \, \int_{\mathbb B_I} \left(|F(z)|^2+|G(z)|^2\right)^{p/2}\,d\mu_I(z)\\
&\leq& \, \int_{\mathbb B_I}\left(2\max\left(|F(z)|^2,|G(z)|^2\right)\right)^{p/2}\,d\mu_I(z)\\
&=&\, \int_{\mathbb B_I}2^{p/2}\max\left(|F(z)|^p,|G(z)|^p\right)\,d\mu_I(z)\\
&\leq& 2^{p/2}\left(\int_{\mathbb B_I} |F(z)|^p\,d\mu_I(z)+\int_{\mathbb B_I} |G(z)|^p\,d\mu_I(z)\right)\\
&\lesssim&\, \parallel F\parallel_{H^p(\mathbb B_I)}^p+\parallel G\parallel_{H^p(\mathbb B_I)}^p\, \leq\, 2\parallel f\parallel_{H^p(\mathbb B)}^p
\end{eqnarray*}
which means that $\mu$ is slice Carleson.
\end{proof}

\begin{theorem}\label{thmequiv} A finite, positive, Borel measure $\mu$ is a Carleson measure for $H^p(\mathbb B)$ if and only if for every $ q=re^{J\theta_0}\in\mathbb B$
\begin{equation}\label{sCeq2}
\mu (S(\theta_0,r))\ \lesssim \ |A_I(\theta_0,r)| \,,
\end{equation}
$|A_I(\theta_0,r)|$ being the lenght of the arc $A_I(\theta_0,r)$, and $I\in\mathbb S$.
\end{theorem}
\begin{proof} The condition of being Carleson and condition \eqref{sCeq2} are both additive, hence we split any measure $\mu$ as the sum of two measures  $\mu=\mu_{\mathbb R}+\tilde{\mu}$ with $\mathrm{supp}\,\mu_{\mathbb R}\subset\mathbb B\cap\mathbb R$ and $\tilde{\mu}(\mathbb B\cap\mathbb R)=0$, and prove the theorem for $\mu_{\mathbb R}$ and $\tilde{\mu}$.\vspace{0.3cm}

Let $\mu_{\mathbb R}$ be a measure with support in $\mathbb B\cap\mathbb R$. Hence $\mu_{\mathbb R}$ is Carleson if and only if it is slice Carleson. Moreover $\mu_{\mathbb R}(S_I(\theta_0,r))=\mu_{\mathbb R}(S(\theta_0,r))$ for any imaginary unit $I$. Hence Proposition \ref{sC} implies the thesis.\vspace{0.3cm}

Let $\tilde{\mu}$ be a measure such that $\tilde{\mu}(\mathbb B\cap\mathbb R)=0$.

 First, we suppose that $\tilde{\mu}$ is a measure such that any symmetric box $S(\theta_0,r)=S(q)$ if $q=re^{I\theta_0}$ has measure controlled by $|A_I(q)|$. Thus, for any $w\in\mathbb B$

\begin{eqnarray*} |A_I(w)|\ &\gtrsim& \tilde{\mu}(S(w)) = \int_{S(w)}d\tilde{\mu}(q)= \int_{\mathbb S}\left(\int_{S_I(w)}d\tilde{\mu}_I^+(z)\right) d\nu(I) \\
&=& \int_{\mathbb S}\left(\int_{S_{J_0}(x+yJ_0)}d\tilde\mu^{+\emph{proj}}_{I}(x+yJ_0)\right) d\nu(I)
\end{eqnarray*}
where $J_0$ is any fixed imaginary unit, $S_{J_0}(w)$ is the projection of the symmetric box $S(w)$ on the fixed $\mathbb B^+_{J_0}$ and $\tilde{\mu}_I^{+\emph{proj}}$ is the projection of the measure $\tilde{\mu}^+_I$ on the same slice:
\begin{eqnarray*}\tilde{\mu}_I^{+\emph{proj}}(E)\ =\ \tilde{\mu}_I^+(\{x+yI\ |\ y>0, \ x+yJ_0\in E\})\\
d\tilde{\mu}_I^{+\emph{proj}}(x+yJ_0)\ =\ d\tilde{\mu}^+_I(x+yI),
\end{eqnarray*}
for any $E\subset\mathbb B^+_{J_0}$. Then, the measure

$$\int_{\mathbb S}d\tilde\mu^{+\emph{proj}}_{I}(x+yJ_0) d\nu(I)$$
is Carleson for $H^p(\mathbb B_{J_0})$. Now, let $f\in H^p(\mathbb B)$. Using the Representation Formula, if $I$ denotes the imaginary unit of $q$, and $J\neq\pm I$ is any imaginary unit, we have

\begin{eqnarray*}
\int_{\mathbb B} &|f(q)|^p&\, d\tilde{\mu}(q) =\\
&=& \int_{\mathbb S}\int_{\mathbb B_I^+}\left|\frac{1-IJ}{2}f(x+yJ)+\frac{1+IJ}2f(x-yJ)\right|^p\,d\tilde{\mu}^+_I(x+yI)\,d\nu(I)\\
&\lesssim& \int_{\mathbb S}\int_{\mathbb B_I^+}\left(|f(x+yJ)|^p+|f(x-yJ)|^p\right)\,d\tilde{\mu}^+_I(x+yI)\,d\nu(I)\\
\end{eqnarray*}
where we have used the fact that the map $x\mapsto x^p$ if convex if $p\geq 1$ and so
\begin{equation}\label{disug_p}
\left|\frac{1-IJ}{2}f(x+yJ)+\frac{1+IJ}2f(x-yJ)\right|^p\leq {2^{p-1}}\left(\left|f(x+yJ)\right|^p+\left|f(x-yJ)\right|^p\right)
\end{equation}
while if $0<p<1$ the map $x\mapsto x^p$ is subadditive on the positive real axis, thus
\begin{equation}\label{disug_p1}
\left|\frac{1-IJ}{2}f(x+yJ)+\frac{1+IJ}2f(x-yJ)\right|^p\leq \left(\left|f(x+yJ)\right|^p+\left|f(x-yJ)\right|^p\right).
\end{equation}
Then we have
\begin{eqnarray*}
\int_{\mathbb B} |f(q)|^p\, d\tilde{\mu}(q) &\lesssim &\int_{\mathbb S}\int_{\mathbb B_I^+}\left(|f(x+yJ)|^p+|f(x-yJ)|^p\right)\,d\tilde{\mu}_I^+(x+yI)\,d\nu(I)\\
&=& \int_{\mathbb B_J^+}|f(x+yJ)|^p\int_{\mathbb S}\,d\tilde{\mu}_I^{+\emph{proj}}(x+yJ)\,d\nu(I) +\\
&+& \int_{\mathbb B_{-J}^+}|f(x-yJ)|^p\int_{\mathbb S}\,d\tilde{\mu}_I^{+\emph{proj}}(x-yJ)\,d\nu(I)\\
&=& \int_{\mathbb B_J^+}|f(x+yJ)|^p\int_{\mathbb S}\,d\tilde{\mu}_I^{+\emph{proj}}(x+yJ)\,d\nu(I) +\\
&+& \int_{\mathbb B_{J}^+}|f(x+yJ)|^p\int_{\mathbb S}\,d\tilde{\mu}_I^{+\emph{proj}}(x+yJ)\,d\nu(I)\\
&=& 2\parallel f\parallel^p_{H^p(\mathbb B_I)}\\
\end{eqnarray*}
and this concludes the proof.
Conversely, suppose $\tilde{\mu}$ is Carleson for $H^p(\mathbb B)$. Consider the function
 $$
 K(q)\ =\ \frac1{4\pi}\,\int_{\mathbb S} k_{u+vI}(q)\,dA(I)
 $$
where $k_{u+vI}(q)=k_w(q)=k(q,w)=(1-q\overline{w})^{-*}$ is the reproducing kernel for $H^2(\mathbb B)$ (see \eqref{kernel}) and $dA(I)$ is the area element on $\mathbb S$. Then, using that $k_w(q)=\overline{k_q(w)}$ for any $w,q\in\mathbb B$, the Representation Formula and the fact that
$\int_{\mathbb S} I \,dA(I)=0$, we have
 \begin{eqnarray*}
K(q) &=&\frac1{4\pi}\,\int_{\mathbb S} k_{u+vI}(q)\,dA(I) = \frac1{4\pi}\,\int_{\mathbb S} \overline{k_{q}(u+vI)}\,dA(I)\\
 &=& \frac1{4\pi}\,\int_{\mathbb S} \overline{\left(\frac{1-IJ}2 k_{q}(u+vJ) +\frac{1+IJ}2 k_{q}(u-vJ\right)}\,dA(I)\\
 &=& \frac1{4\pi}\left(\int_{\mathbb S} \overline{k_{q}(u+vJ)}\frac{1-JI}2\,dA(I)+\int_{\mathbb S} \overline{k_{q}(u-vJ)}\frac{1+JI}2\,dA(I)\right)\\
&=& \overline{k_{q}(u+vJ)}\left(\frac 12-\frac{J\int_{\mathbb S}I\,dA(I)}{8\pi}\right)+\overline{k_{q}(u-vJ)}\left(\frac 12+\frac{J\int_{\mathbb S}I\,dA(I)}{8\pi}\right)\\
&=& \frac12 (\overline{k_{q}(u+vJ)}+\overline{k_{q}(u-vJ)})\ =\ \frac12 (k_{u+vJ}(q)+k_{u-vJ}(q))\\
&=& \frac12\left((1-q\overline w)^{-*}+(1-qw)^{-*}\right).
\end{eqnarray*}
From this formula we deduce
\begin{eqnarray*}
K(\overline q) &=& \frac12\left((1-\overline q\, \overline w)^{-*}+(1-\overline qw)^{-*}\right)\\
&=& \frac12\left( (1- 2\overline q {\rm Re}(\overline w) +\overline q^2 |\overline w|^2)^{-1}(1-\overline q w)+ (1- 2\overline q {\rm Re}(w) +\overline q^2 |w|^2)^{-1}(1-\overline q \overline w)\right)\\
&=& \frac 12 (1- 2\overline q {\rm Re}(w) +\overline q^2 |w|^2)^{-1}(2-\overline q (w+\overline w))
\end{eqnarray*}
and
\begin{eqnarray*}
\overline{K(q)} &=& \frac12\left(\overline{(1- q \overline w)^{-*}+(1- qw)^{-*}}\right)\\
&=& \frac12\left( \overline{(1- 2 q {\rm Re}(\overline w) + q^2 |\overline w|^2)^{-1}(1- q w)+ (1- 2 q {\rm Re}(w) + q^2 |w|^2)^{-1}(1- q \overline w)}\right)\\
&=& \frac 12 (2-(\overline w+ w)\overline q )(1- 2\overline q {\rm Re}(w) +\overline q^2 |w|^2)^{-1}\\
&=& \frac 12 (1- 2\overline q {\rm Re}(w) +\overline q^2 |w|^2)^{-1} (2-\overline q ( w+ \overline w)),
\end{eqnarray*}
where we used that $w+\overline w\in\mathbb R$. Thus we have $K(\overline q)=\overline{K(q)}$ and the function $K(q)$ is quaternionic
intrinsic. Note that $K(q)$ is the sum of functions in $H^2(\mathbb B)$ and, as such, it belongs to $H^2(\mathbb B)$.
Moreover, it is a slice regular function in the ball, bounded, and never vanishing. Note that for any positive real number $\nu$ we define $q^\nu={\rm ext}(z^\nu)$ which is everywhere defined. So the $\frac 2p$-power $K^{2/p}$ is defined for any $p$ and it is a quaternionic intrinsic bounded slice regular function in the ball, hence it is a function in $H^p(\mathbb B)$.

Thanks to the fact that $\tilde{\mu}$ is Carleson for $H^p(\mathbb B)$, we get that
$$\int_{\mathbb B}|K^{2/p}(q)|^p\,d\tilde{\mu}\ \lesssim\ \parallel K^{2/p}\parallel_p^p$$
and thus the following chain of estimates:
\begin{eqnarray}
\int_{\mathbb B}|K^{2/p}(q)|^p\,d\tilde{\mu} &\lesssim& \parallel K^{2/p}\parallel_p^p\\
\nonumber&=& \sup_{I\in\mathbb S} \lim_{r\to 1^-} \frac{1}{2\pi}\int_0^{2\pi}|K^{2/p}(re^{I\theta})|^p\,d\theta \\
\nonumber & = & \sup_{I\in\mathbb S} \lim_{r\to 1^-} \frac{1}{2\pi}\int_0^{2\pi}|K(re^{I\theta})|^2\,d\theta \\
\nonumber&=& \parallel K\parallel_2^2\ \leq\ \frac1{1-|w|}.
\end{eqnarray}
Moreover
\begin{eqnarray}
\int_{\mathbb B}|K^{2/p}(q)|^p\,d\tilde{\mu} &\gtrsim& \int_{S(w)}|K^{2/p}(q)|^p\,d\tilde{\mu}\\
\nonumber&=& \int_{S(w)}|K|^2\,d\tilde{\mu} \geq \int_{S(w)}\frac1{(1-|w|^2)^2}\,d\tilde{\mu}\\
\nonumber&=& \frac{\tilde{\mu}(S(w))}{(1-|w|^2)^2},
\end{eqnarray}
where we used the fact that for all $q\in \mathbb B$ (and thus also for all $q\in S(w)$) the inequality $|K(q)|\geq\frac1{1-|w|^2}$ holds.

We conclude that
$$\tilde{\mu}(S(w)) \lesssim \frac{(1-|w|^2)^2}{1-|w|}\leq 4(1-|w|)$$
and the assertion follows.
\end{proof}

\section{Carleson measures in Bergman spaces}

The Bergman space $\mathcal A^2(U)$ has been studied in a series of papers, see \cite{CGLSS}, \cite{CGS2}, \cite{CGS3}, \cite{CGS4}. In \cite{CCGG} the authors study the weighted $\mathcal A^p(U)$ Bergman spaces and, as a special case, we have the following definition:
\begin{definition}
Let $U\subseteq\mathbb H$ be an axially symmetric s-domain.
For $p > 0$, we define the  slice regular Bergman space $\mathcal A^p(U)$ as the
quaternionic right linear space of all slice regular functions on $U$ such that
\begin{equation} \sup_{I\in\mathbb S}\int_{U\cap\mathbb C_I}|f(z)|^p\,d \lambda(z) < \infty ,
\end{equation}
where $\lambda$ is the Lebesgue measure on $\mathbb C_I$.
For $p > 0$, and any fixed $I\in\mathbb S$, we define the  slice regular Bergman space $\mathcal A^p(U_I)$ as the
quaternionic right linear space of all slice regular functions on $U_I=U\cup\mathbb C_I$ such that
\begin{equation} \int_{U\cap\mathbb C_I}|f(z)|^p\,d \lambda(z)<\infty .
\end{equation}
\end{definition}
As it is proved in \cite{CCGG}, the slice regular Bergman space $\mathcal A^p(U)$ can be equipped with the norm
\begin{equation}
\|f\|_{\mathcal A^p(U)}=\left(\sup_{I\in\mathbb S}\int_{U\cap\mathbb C_I}|f(z)|^p\,d \lambda(z)\right)^{\frac 1p} .
\end{equation}
From now on, we will consider $U=\mathbb B$, and the norm in $\mathcal A^p(\mathbb B)$ will be denoted by $\|\cdot \|_{\mathcal A^p}$.

In this section, we will characterize Carleson measures for the Bergman spaces.  These measures are defined as follows:
\begin{definition}
Let $\mu$ be a finite, positive, Borel measure on $\mathbb B$.
We say that $\mu$ is a Carleson measure for $\mathcal A^p(\mathbb B)$ if for any $f\in\mathcal A^p(\mathbb B)$
the inequality
\begin{equation}
\int_{\mathbb B} |f(z)|^p\, d\mu \lesssim \| f\|_{\mathcal A^p}^p
\end{equation}
holds.
We say that $\mu$ is a slice Carleson measure for $\mathcal A^p(\mathbb B)$ if for any $f\in\mathcal A^p(\mathbb B)$
the inequality
\begin{equation}\label{sliceCAp}
\int_{\mathbb B_I} |f(z)|^p\, d\mu_I \lesssim \| f\|_{\mathcal A^p}^p
\end{equation}
holds, where $\mu_I=\mu_{\mathbb R}+\tilde\mu_I^++\tilde\mu_{-I}^+$.
\end{definition}
The argument used in the proof of Proposition \ref{sliceCarleson} yields the following:
\begin{proposition}
If a measure $\mu$ is slice Carleson for $\mathcal A^p(\mathbb B)$ then it is also Carleson.
\end{proposition}
In \cite{AS1} the authors define a pseudohyperbolic metric using the slice regular Moebius transformation:
\begin{equation}\label{Mobius}
\rho(w,\alpha)=|(1- q\overline{\alpha})^{-*}*(q-\alpha)|_{|q=w}=|(1- q\overline{w})^{-*}*(q-w)|_{|q=\alpha}.
\end{equation}
If $\alpha\in\mathbb C_I$ we can consider the restriction $\rho_I$ to $\mathbb C_I$. Then $\rho_I$ reduces to the pseudohyperbolic distance on the slice $\mathbb C_I$:
\[
\rho_I(z,\alpha)=\left|\frac{z-\alpha}{1-z\bar\alpha}\right|.
\]
We define the pseudohyperbolic disc (in $\mathbb C_I$) with center $\alpha$ and radius $r$ as
\[
\Delta_I(\alpha,r)=\{z\in\mathbb C_I\ : \ \rho_I(z,\alpha)<r\},
\]
and
\[
\Delta(\alpha,r)=\{q=x+Jy\in\mathbb H\ : z=x+Iy\in \Delta_I(\alpha,r)\}.
\]
Note that $\Delta(\alpha,r)$ is the axially symmetric completion of $\Delta_I(\alpha,r)$, i.e.
\[
\Delta(\alpha,r)=\underset{I\in\mathbb S}{\bigcup} \Delta_I(\alpha,r).
\]
Let us  denote by $| \Delta_I(\alpha,r)|$ the area of $ \Delta_I(\alpha,r)$. Since $ \Delta_I(\alpha,r)$ is a pseudohyperbolic disc in $\mathbb C_I$, we have
\begin{equation}\label{areaD}
| \Delta_I(\alpha,r)|= \pi \frac{r^2(1-|\alpha|^2)^2}{(1- r^2 |\alpha|^2)^2} .
\end{equation}
Moreover, we recall the following results (see e.g. \cite{DS}) stated using our notations:
\begin{proposition}\label{pro_lemma}
For any fixed $I\in\mathbb S$ and for $r\in (0,1)$ there exist a sequence of points $\{\alpha_k\}\subset \mathbb B_I$ and an integer $n_0$ such that
\begin{equation}\label{alfakappa}
\underset{k=1}{\overset{\infty}\bigcup} \Delta_I(\alpha_k, r)=\mathbb B_I.
\end{equation}
Moreover, no point $z\in\mathbb B_I$ belongs to more than $n_0$ discs $\Delta_I(\alpha_k, R)$, where $R=\frac 12(1+r)$.
\end{proposition}
\begin{proposition}
Let $\mu$ be a finite, positive, Borel measure on $\mathbb B_I$. Then $\mu_I(\Delta_I(\alpha,r))\leq C | \Delta_I(\alpha,r)|$ if and only if $\mu_I(S_I(z))\leq |A_I(z)|^2$.
\end{proposition}
We now prove a result characterizing slice Carleson measures.
\begin{theorem}
Let $\mu$ be a finite, positive, Borel measure on $\mathbb B_I$ and let $p\in(0,\infty)$ be fixed. The following are equivalent:
\begin{enumerate}
\item The measure $\mu$ is slice Carleson for $\mathcal A^p(\mathbb B)$.
\item The inequality $\mu_I(\Delta_I(\alpha,r))\leq C | \Delta_I(\alpha,r)|$ holds for all $r\in (0,1)$ for some constant $C$ depending on $r$ only, and for all $\Delta_I(\alpha,r)$, $\alpha\in\mathbb B_I$.
\item The inequality $\mu_I(\Delta_I(\alpha,r))\leq C | \Delta_I(\alpha,r)|$ holds for some $r\in (0,1)$ for some constant $C$ depending on $r$ only, and for all $\Delta_I(\alpha,r)$, $\alpha\in\mathbb B_I$.
\end{enumerate}
\end{theorem}
\begin{proof}
 (1)$\Longrightarrow$(2). Let us assume that
$\mu$ is slice Carleson for $\mathcal A^p(\mathbb B)$.Then \eqref{sliceCAp} holds in particular for the classical
Bergman space $\mathcal A^p(\mathbb D)$ identified with a subset of $\mathcal A^p(\mathbb B_I)$ for some fixed $I\in\mathbb S$. Thus (2) follows from the complex case.\\
(2)$\Longrightarrow$(3) is trivial, so we show that (3)$\Longrightarrow$(1). Let $I\in\mathbb S$ be fixed and let us write $f(z)=F(z)+G(z)J$, where $F,G$ are holomorphic functions. We know that the statement holds for $F,G$, because it holds in the complex case, thus
\[
\begin{split}
&\int_{\mathbb B_I}|F(z)|^p d\mu_I \lesssim \int_{\mathbb B_I} |F(z)|^p\, d\lambda\\
&\int_{\mathbb B_I}|G(z)|^p d\mu_I \lesssim \int_{\mathbb B_I} |G(z)|^p\, d\lambda
\end{split}
\]
from which, reasoning as in the proof of Proposition \ref{sC}, we have
\[
\begin{split}
&\int_{\mathbb B_I}|f(z)|^p d\mu_I = \int_{\mathbb B_I} (|F(z)|^2+|G(z)|^2)^{p/2}\, d\mu_I\\
& \leq 2^{p/2} \left(\int_{\mathbb B_I} |F(z)|^p\, d\mu_I +\int_{\mathbb B_I} |G(z)|^p\, d\mu_I\right)\\
& \lesssim  \left(\int_{\mathbb B_I} |F(z)|^p\, d\lambda +\int_{\mathbb B_I} |G(z)|^p\, d\lambda\right)\\
& \leq \| F\|_{\mathcal A^p}^p+ \| G\|_{\mathcal A^p}^p\leq 2  \| f\|_{\mathcal A^p}^p.
\end{split}
\]
\end{proof}
Next result characterizes Carleson measures in terms of the axially symmetric completion of a pseudohyperbolic disc.
\begin{theorem}\label{tubo}
Let $\mu$ be a positive, finite measure on $\mathbb B$ and let $p\in(0,\infty)$ be fixed. The following are equivalent:
\begin{enumerate}
\item The measure $\mu$ is Carleson for $\mathcal A^p(\mathbb B)$.
\item The inequality $\mu(\Delta(\alpha,r))\leq C | \Delta_I(\alpha,r)|$ holds for all $r\in (0,1)$ for some constant $C$ depending on $r$ only, and for all $\Delta_I(\alpha,r)$, $\alpha\in\mathbb B_I$.
\item The inequality $\mu(\Delta(\alpha,r))\leq C | \Delta_I(\alpha,r)|$ holds for some $r\in (0,1)$ for some constant $C$ depending on $r$ only, and for all $\Delta_I(\alpha,r)$, $\alpha\in\mathbb B_I$.
\end{enumerate}
\end{theorem}
\begin{proof}
(1)$\Longrightarrow$(2). Let us assume that
$\mu$ is Carleson for $\mathcal A^p(\mathbb B)$. As in the proof of Theorem \ref{thmequiv}, we can assume
that $\mu(\mathbb B\cap\mathbb R)=0$.\\
Let us set $$
H(q)\ =\ \frac1{4\pi}\,\int_{\mathbb S} h_{u+vI}(q)\,dA(I)
$$
where
$$
h_{u+vI}(q)=h_w(q)=(1-q\overline{w})^{-2*}=(1-2\bar q\bar w+\bar q^2\bar w^2)(1-2{\rm Re}(w)\bar q+ | w|^2\bar q^2)^{-2}
$$
is, modulo the factor $\frac{1}{\pi}$, the Bergman kernel for $\mathcal A^2(\mathbb B)$, see \cite{CGS1}, \cite{CGS4}. Reasoning as in the second part of the proof of Theorem \ref{thmequiv}, and using the fact that $\overline{h_w(q)}=h_q(w)$ we have
that
\begin{equation}\label{hqw}
H(q)=\frac1{4\pi}\,\int_{\mathbb S} h_{u+vI}(q)\,dA(I)=\frac{1}{2}\left(h_{u+Iv}(q)+h_{u-Iv}(q)\right)
\end{equation}
and so $H(q)$ belongs to $\mathcal A^2(\mathbb B)$ since it is superposition of functions in $\mathcal A^2(\mathbb B)$. Moreover, we have
$H(\bar q)=\overline{H(q)}$ and so $H(q)$ is quaternionic intrinsic, nonvanishing in $\mathbb B$. By its definition, it is clear that the function $H(q)$ depends also on the variables $u,v$, even though this dependence is not explicitly stated.
The Representation formula implies
\begin{equation}\label{womega}
\begin{split}
h_w(q)+h_{\bar w}(q)&=\frac 12 \Big( h_\omega(q)+h_{\bar \omega}(q) +(h_{\bar \omega}(q))-h_\omega(q))\,II_w\\
&+ h_{\bar \omega}(q)+h_{\omega}(q) + (h_{\omega}(q))-h_{\bar \omega}(q))\, II_w)\Big)\\
&=h_\omega(q)+h_{\bar \omega}(q)
\end{split}
\end{equation}
for any $w, \omega$ belonging to the same sphere.
\\
For any $p>0$ we can consider $H^{2/p}(q)$ (see Proposition \ref{comp}) which is a function in $\mathcal A^p(\mathbb B)$. Since $\mu$ is Carleson for $\mathcal A^p(\mathbb B)$ we have
\begin{eqnarray}\label{13}
\int_{\mathbb B}|H^{2/p}(q)|^p\,d\mu &\lesssim& \parallel H^{2/p}\parallel_{\mathcal A^p(\mathbb B)}^p\\
\nonumber&=& \sup_{I\in\mathbb S} \int_{\mathbb B_I} |H^{2/p}(q)|^p\,d\lambda \\
\nonumber&=& \sup_{I\in\mathbb S} \int_{\mathbb B_I} |H(q)|^2\,d\lambda \\
\nonumber&=& \sup_{I\in\mathbb S} \int_{\mathbb B_I} |h_w(q)+h_{\bar w}(q)|^2\,d\lambda \\
\nonumber& \lesssim & \frac{1}{(1-(u^2+v^2))^2}
\end{eqnarray}
where last inequality is due to the fact that by \eqref{womega} we can choose $w$ on the same complex plane as $q$ and thus the inequality is obtained from the analogous inequality in the complex case.
\\
On the other hand,
\begin{eqnarray}\label{14}
\int_{\mathbb B}|H^{2/p}(q)|^p\,d\mu & \gtrsim & \int_{\Delta(w,r)}|H^{2/p}(q)|^p\,d\mu\\
\nonumber&=& \int_{\Delta(w,r)}|H(q)|^2\,d\mu \\
\nonumber&\geq &  \int_{\Delta(w,r)} \frac1{(1-|w|^2)^2} \,d\mu\\
\nonumber&\geq & \mu(\Delta(w,r)) \frac{(1-r|w|)^4}{(1-|w|^2)^4}.
\end{eqnarray}
Using the hypothesis, we have
\begin{eqnarray}
\nonumber& & \mu(\Delta(w,r)) \frac{(1-r|w|)^4}{(1-|w|^2)^4}\\
\nonumber &\leq & \int_{\mathbb B}|H^{2/p}(q)|^p\,d\mu \\
\nonumber&\lesssim & \| H^{2/p} \|^p_{\mathcal{A}^p} .\\
\end{eqnarray}
Using \eqref{13} and \eqref{14} we obtain
\[
\begin{split}
\mu(\Delta(w,r))&\lesssim \| H\|^{p}_p \frac{(1-|w|^2)^4}{(1-r|w|)^4}\lesssim \frac{(1-|w|^2)^2}{(1-r|w|)^4}\\
&\lesssim |\Delta_I(w,r)| \frac{1}{r^2(1-r|w|)^2} \leq |\Delta_I(w,r)| \frac{1}{r^2(1-r)^2}
\end{split}
\]
and the assertion follows.\\
The fact that (2) implies (3) is obvious, so we show that (3) implies (1). Thus we assume that (3) is in force for some $r\in (0,1)$.
Then we take the axially symmetric completion of both sides of \eqref{alfakappa}, so that we obtain
$$
\mathbb B = \underset{k=1}{\overset{\infty}\bigcup} \Delta (\alpha_k, r) ,
$$
from which it follows that, if $z\in\mathbb C_J$ and $q\in\mathbb C_I$:
\begin{equation}\label{ineqBergman}
\begin{split}
\int_{\mathbb B} |f(q)|^p\, d\mu &\leq \sum_{k=1}^\infty \int_{\Delta (\alpha_k, r)} |f(q)|^p\, d\mu \\
&=\sum_{k=1}^\infty \int_{\Delta (\alpha_k, r)} \left|\frac{1-IJ}{2} f(z)+ \frac{1+IJ}{2}
f(\bar z)\right|^p \, d\mu \\
&\leq C \sum_{k=1}^\infty  \int_{\Delta (\alpha_k, r)} \left( |f(z)|^p+
|f(\bar z)|^p\right) \, d\mu ,\\
\end{split}
\end{equation}
where we used the Representation Formula and \eqref{disug_p}, \eqref{disug_p1} so that $C=1$ or $C=2^{p-1}$.
Lemma 13 in \cite{DS} yields that for all $z\in \Delta_J (\alpha_k, r)$ the following inequality holds:
\begin{equation}\label{lemma13}
|f(z)|^p \leq \frac{4(1-R)^{-4}}{|\Delta_J(\alpha_k, R)|} \int_{\Delta_J(\alpha_k, R)} |f(s)|^p\, d\lambda (s), \qquad R=\frac 12 (1+r).
\end{equation}
Note that if $z\in \Delta_J(\alpha_k, R)$ then $\bar z\in \Delta_J(\overline{\alpha}_k, R)$
thus, continuing the computations in \eqref{ineqBergman}, and using \eqref{lemma13} we obtain
\[
\begin{split}
C &\sum_{k=1}^\infty \int_{\Delta (\alpha_k, r)} \left( |f(z)|^p+
|f(\bar z)|^p\right) \, d\mu \\
\leq C &\sum_{k=1}^\infty\int_{\Delta (\alpha_k, r)}\left(
\frac{4(1-R)^{-4}}{|\Delta_J(\alpha_k, R)|} \int_{\Delta_J (\alpha_k, R)} |f(s)|^p\, d\lambda +
\frac{4(1-R)^{-4}}{|\Delta_J(\overline{\alpha}_k, R)|} \int_{\Delta_J (\overline{\alpha}_k, R)} |f(s)|^p\, d\lambda\right)d\mu\\
\leq C & \sum_{k=1}^\infty  \mu(\Delta(\alpha_k,r))\left(\frac{4(1-R)^{-4}}{|\Delta_J(\alpha_k, R)|}\int_{\Delta_J ({\alpha}_k, R)} |f(s)|^p\,d\lambda +\frac{4(1-R)^{-4}}{|\Delta_J(\overline{\alpha}_k, R)|} \int_{\Delta_J (\overline{\alpha}_k, R)} |f(s)|^pd\lambda\right)\\
\leq C_1 & \sum_{k=1}^\infty  |\Delta_J(\alpha_k,r)|\left(\frac{4(1-R)^{-4}}{|\Delta_J(\alpha_k, R)|}\int_{\Delta_J ({\alpha}_k, R)} |f(s)|^p\,d\lambda +\frac{4(1-R)^{-4}}{|\Delta_J(\overline{\alpha}_k, R)|} \int_{\Delta_J (\overline{\alpha}_k, R)} |f(s)|^pd\lambda\right),
\end{split}
\]
where we obtained last inequality using our assumption in point (3).
Formula \eqref{areaD} yields
\[
\frac{|\Delta_J(\alpha_k,r)|}{|\Delta_J(\alpha_k, R)|}=\frac{r^2(1-R^2|\alpha|^2)^2}{R^2(1-r^2|\alpha|^2)^2}\leq \frac{1}{(1-r^2|\alpha|^2)^2}\leq \frac{1}{(1-R)^2}
\]
and similarly
\[
\frac{|\Delta_J(\alpha_k,r)|}{|\Delta_J(\overline{\alpha}_k, R)|}\leq \frac{1}{(1-R)^2}.
\]
Thus, using these inequalities, Proposition \ref{pro_lemma} and the fact that $\Delta_J(\alpha_k,r)\subset \mathbb B_J$ for any $\alpha_k$, we have
\[
\begin{split}
C_1  \sum_{k=1}^\infty  &|\Delta_J(\alpha_k,r)|\left(\frac{4(1-R)^{-4}}{|\Delta_J(\alpha_k, R)|}\int_{\Delta_J (\overline{\alpha}_k, R)} |f(s)|^p\,d\lambda +\frac{4(1-R)^{-4}}{|\Delta_J(\overline{\alpha}_k, R)|} \int_{\Delta_J (\overline{\alpha}_k, R)} |f(s)|^pd\lambda\right)\\
&\leq C_1  \sum_{k=1}^\infty \frac{4}{(1-R)^{6}}\left(\int_{\Delta_J (\alpha_k, R)} |f(s)|^p\,d\lambda + \int_{\Delta_J (\overline{\alpha}_k, R)} |f(s)|^p\, d\lambda\right)\\
&\leq C_1   \frac{4n_0}{(1-R)^{6}} \left(\int_{\mathbb B_J} |f(s)|^p\,d\lambda + \int_{\mathbb B_J} |f(s)|^p\, d\lambda\right)\\
&\leq C_1    \frac{4n_0}{(1-R)^{6}} \sup_{J\in\mathbb S}\left(\int_{\mathbb B_J} |f(s)|^p\,d\lambda + \int_{\mathbb B_J} |f(s)|^p\, d\lambda\right)\\
&\lesssim    \|f\|_{\mathcal{A}^p(\mathbb B)}^p.
\end{split}
\]
Following the chain of inequalities we have proved that
$$
\int_{\mathbb B} |f(q)|^p\, d\mu \leq \tilde C \|f\|_{\mathcal{A}^p(\mathbb B)}^p
$$
and thus (1) holds.
\end{proof}
\section{Carleson measures and pseudohyperbolic balls}

In this section, let $\alpha=a+Ib\in\mathbb B$ and by $d=d(\alpha,\partial\mathbb B)$ denote the Euclidean distance of $\alpha$ to the boundary of the unit ball.

Let $\rho$ be the pseudohyperbolic distance defined in (\ref{Mobius}) and $B(\alpha ,r)$ be the pseudohyperbolic ball of center $\alpha$ and pseudohyperbolic radius $r$, i.e.
$$B(\alpha,r)\ =\ \{q\in\mathbb B\ |\ \rho(q,\alpha)<r\}\,.$$

The result in Theorem \ref{tubo} is essentially a result on the axially symmetric completion of a pseudohyperbolic disc in a complex plane. A similar result for pseudohyperbolic balls is false. In order to show that, some estimations on the volume and shape of a pseudohyperbolic ball are needed. These cannot be obtained by straightforward generalizing the complex case, since a composition with a Moebius transformation is involved. Thus we will follow the strategy in \cite{AbS}, where a more intrinsic approach is used.

\begin{lemma}\label{lemma2.2} There exists a constant $C_1$ such that for every $\alpha \in\mathbb B$ and $r\in(0,1)$
$$\forall q\in B(\alpha ,r)\ \ \ \frac{1-r}{C_1}\, d\, \leq\, d(q,\partial\mathbb B)\,\leq\, \frac{C_1}{1-r}\, d\,.$$
\end{lemma}
\begin{proof} Observe that
$$B(\alpha ,r)\cap \mathbb C_I \ =\ \Delta_I(\alpha ,r)$$
and this is one of the connected components of $\Delta(\alpha ,r)\cap \mathbb C_I$.
\\
Let $J\in\mathbb S$ be any imaginary unit. Then
$$\rho(\alpha , x+yJ) \ \geq\ \rho(\alpha ,x+yI)$$
so we have that
$$B(\alpha ,r)\ \subset\ \Delta(\alpha ,r)\,.$$
Hence, if $q=q_1+q_2J\in B(\alpha ,r)$, also $q'=q_1+q_2I\in B(\alpha ,r)$ and $d(q,\partial\mathbb B)= d(q',\partial \mathbb B)$. Thus it is sufficient to prove the statement for all $q'\in B(\alpha ,r)\cap \mathbb C_I$. But this is the corresponding statement for the complex case, see e.g. \cite{AbS}.
\end{proof}

\begin{lemma}\label{d-lemma}Let $\alpha\in\mathbb B_I$. There exist two positive constant $C_2(r),c_2(r)$, depending only on $r$ such that:
$$c_2(r) d^4\ \leq\ |\Delta_I(\alpha,r)|\ \leq C_2(r) d^4$$
\end{lemma}

\begin{proof} As already remarked, $\Delta_I$ is (see \cite{AS}) a disc of radius $r_1\ =\ r\frac{1-| \alpha |^2}{1-r^2| \alpha |^2}$ centered at the point $q_1\ =\ \frac{1-r^2}{1-r^2| \alpha |^2}\ \alpha$.

Notice that
$$|\Delta_I|\ =\ \pi r_1 ^2\ =\ \pi r^2\frac{(1-| \alpha |^2)^2}{(1-r^2| \alpha |^2)^2}\ = \ \pi r^2\frac{d^4}{(1-r^2(1-d^2)^2)^2}$$

Thus
$$\pi r^2 d^4\leq |\Delta_I|\ \leq \ \pi r^2\frac{d^4}{(1-r^2)^2}\,,$$
since $0<d\leq1$.
\end{proof}

Let $\eta$ be the Lebesgue $4$-dimensional measure normalized so that $\eta(\mathbb B)=1$, and let us now estimate the $\eta$-volume of a pseudohyperbolic ball:

\begin{lemma}\label{lemma2.1} For every $r\in(0,1)$ there exist two constants $C_3(r)$ and $c_3(r)$ such that for every $\alpha \in\mathbb B$
$$c_3(r)\,d^8\ \leq\ \eta( B(\alpha ,r))\ \leq\ C_3(r)\,d^8\,.$$
\end{lemma}
\begin{proof} To get such estimates on the $\eta$-volume of a pseudohyperbolic ball $B(\alpha ,r)$, we will find two domains $B_1(\alpha ,r),B_2(\alpha ,r)$ such that
$$B_1(\alpha ,r) \ \subset  B(\alpha ,r)\ \subset\ B_2(\alpha ,r)$$
and then estimate the $\eta$-volume of these.\vspace{0.3cm}

\noindent
{\it Estimate from above}. We will divide the construction of $B_2(\alpha,r)$ in two cases. Let first $\alpha$ be such that $\rho(\alpha,a)< r$. Then we set
$$B_2(\alpha,r)\ =\ B(a,2r)\ \supset\ B(\alpha,r)\,,$$
and $B(a,2r)$ is simply a Euclidean ball of Euclidean radius (see proof of the previous lemma, or \cite{AS})
$$R\ =\ 2r\frac{1-|a |^2}{1-4r^2| a |^2} \ \leq \ \frac{2r}{1-4r^2}(d^2+b^2)\,.$$
Since $|b|$ is the Euclidean distance between $a$ and $\alpha$, and $a\in \Delta_I(\alpha,r)\subset B(\alpha,r)$, one gets in the same way that
$$|b|\ < \ 2 \cdot \emph{radius}(B(\alpha,r))\ < \ \left(\frac{2r}{1-r^2}\right)d^2\, .$$
Hence
$$R\ \leq\ \frac{2r}{1-4r^2}d^2\left(1+d^2\frac{4r^2}{(1-r^2)^2}\right)\ \leq\ C_4(r)d^2\,.$$
From these inequalities, the desired estimate on the $\eta$-volume follows immediately.\vspace{0.2cm}

\noindent
Second case: $\rho(\alpha,a)\geq r$. Notice that in this case $B(\alpha ,r)\subset B_2(\alpha ,r)$, where
$$B_2(\alpha ,r) \ =\ \{q=x+Jy\in\mathbb B\ |\ \rho(x+Iy,\alpha )<r,\ \rho(a+Jb,\alpha )<r\}\,.$$

Since $B_2(\alpha, r)\cap\mathbb R=\emptyset$, the volume of $B_2(\alpha ,r)$ can be computed using Pappus-Guldinus theorem and thus it is given - up to a renormalization constant - by the area of $\Delta_I$ (estimated in Lemma \ref{d-lemma}) times the area of the set $\Gamma= B_2(\alpha, r)\cap [q_1]$ where $q_1=\ \frac{1-r^2}{1-r^2| \alpha |^2}\ \alpha$ is the Euclidean center of $\Delta_I$.

Let $q$ be any element in $[\alpha]$. Then $q=a+I_q b$ and, by taking a suitable $J\perp I$, we can write $I_q=\cos\theta I+\sin\theta J$, so that $q=a+b(\cos\theta I+\sin\theta J)$.  By direct computation of the pseudohyperbolic distance $\rho(q,\alpha )$, one gets that $q=a+b(\cos \theta I+\sin\theta J)\in B_2(\alpha ,r)$ if and only if
$$\frac{x^2}{d^4+x^2}\ < r\,,$$
where $x$ is the Euclidean distance between $q$ and $\alpha $, i.e.\
\begin{equation}\label{stimadistanza}x^2\ < \ \frac{d^4r}{1-r}\,.\end{equation}

\begin{figure}[h!]
\centering
 \includegraphics[width=6in,height=3.0in,keepaspectratio]{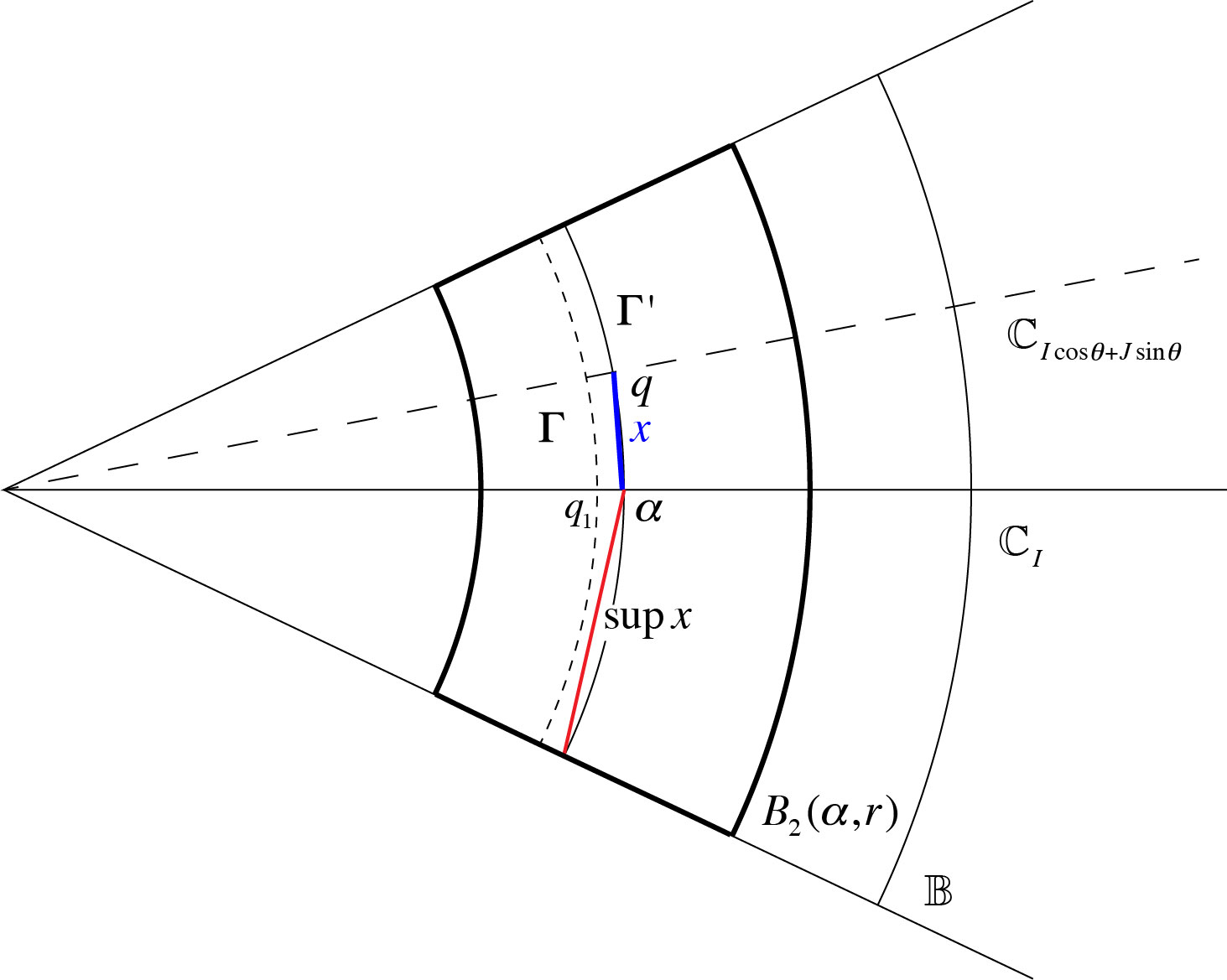}
\caption{}
\label{prep1}
\end{figure}

The area $|\Gamma|$ of $\Gamma$
can be estimated by:

$$|\Gamma|\ \leq\ 2\pi (\sup x^2) \frac{| q_1|^2}{| \alpha |^2}\ =\ 2\pi \frac{d^4r}{1-r}\cdot\frac{1-r^2}{1-r^2| \alpha |^2} \ =\ 2\pi d^4r\frac{1+r}{1-r^2(1-d^2)}\ \leq 2\pi d^4r\frac{1+r}{1-r^2}$$

Thus
$$\eta( B(\alpha ,r))\ \leq C_3(r) d^8$$
as claimed.

\vspace{0.3cm}

\noindent
{\it Estimate from below}. Let us consider again the intersection
$$\Delta_I\ =\  B(\alpha ,r)\cap\mathbb C_I\,.$$
As already observed, $\Delta_I$ is an Euclidean disc of radius $r_1\ =\ r\frac{1-| \alpha |^2}{1-r^2| \alpha |^2}$ centered at the point $q_1\ =\ \frac{1-r^2}{1-r^2| \alpha |^2}\ \alpha$.

The pseudohyperbolic ball $B_1(\alpha ,r)$ with same center and radius of $\Delta_I$ is contained in $B(\alpha ,r)$ and it is a Euclidean ball of radius $r_1$.

Hence, since $\eta(\mathbb B)=1$, the $\eta$-volume of a ball of radius $r_1$ is $r_1^4$.
\begin{eqnarray}\nonumber \eta(B(\alpha ,r)) &\geq& \eta(B_1(\alpha ,r))\ =\ r^4\left(\frac{1-| \alpha |^2}{1-r^2| \alpha |^2}\right)^4 \ \geq\\
\nonumber &\geq& r^4 (1-| \alpha |^2)^4\ =\ r^4 d^8\,.
\end{eqnarray}

\end{proof}

Now we give an estimate on the number of pseudohyperbolic balls needed to cover $\Delta(\alpha, r)$. Recall that by $[\alpha]$ we denote the 2-sphere of points with same real part and same modulus as $\alpha$.

\begin{lemma}\label{covering-tubes} There exist two constants $c_5(r)$ and $C_5(r)$ such that
\begin{enumerate}
\item $\Delta(\alpha,r)$ can be covered with at most $C_5(r)d^{-4}$ balls $B(\alpha_j,4r)$ such that all $\alpha_j\in[\alpha]$;
\item if $\alpha=Iy$, then $\Delta(Iy,r)$ contains at least $c_5(r)d^{-4}$ disjoint balls $B(J_jy,r)$.
\end{enumerate}
\end{lemma}

\begin{proof} (1) If $\alpha=a+Ib$ is such that $\rho(\alpha,a)<r$, then $\alpha\in B(a,r)$ and $B(\alpha,r)\subset B(a,2r)$. Thus it holds $B(\alpha,4r)\supset\Delta(\alpha,r)$.\\

Second case: if $\rho(\alpha,a)\geq r$, then we can choose $C_5(r)d^{-4}$ points $\alpha_j\in[\alpha]$ such that any point $\beta\in[\alpha]$ has the square of the Euclidean distance from at least one of the points $\alpha_j$ less that
$$\frac{r}{1-r}d^4\,.$$
This can be done, indeed, if we choose $C_5(r)d^{-4}$ imaginary units $I_j\in\mathbb S$ with the property that any imaginary unit $J\in \mathbb S$ has the square of the Euclidean distance from at least one of the points $I_j$ less that $\frac{r}{1-r}d^4$, defining $\alpha_j=a+I_jb$ one has the desired property. Thus $$\min\rho(\beta,\alpha_j) \ <\ r\,,$$
thanks to (\ref{stimadistanza}).

Notice that for any point $\gamma=x+Jy\in\Delta(\alpha,r)$ then for $\beta=a+Jb\in[\alpha]$ there exists $\alpha_j$ such that
$$\rho(\gamma,\alpha_j)\ \leq\ \rho(\gamma,\beta)+\rho(\beta,\alpha_j)\ < \ 2r\,$$
thus proving the thesis.\vspace{0,3cm}\\

(2) If $d^2\leq\frac12$, then the inequality holds trivially with $c_5(r)=\frac14$. So we can suppose $d^2>\frac12$. As remarked in the proof of Lemma \ref{lemma2.1} the center of each $B(J_jy,r)$ spans an area $|\Gamma|$
$$|\Gamma|\ \leq 2\pi d^4r\frac{1+r}{1-r^2}\,,$$
while the area spanned by the center of $\Delta(Iy,r)$ is $|T|$
$$|T|\,=\,4\pi\frac{(1-r^2)^2}{(1-r^2(1-d^2))^2}(1-d^2)$$
Estimating $1-d^2$ with $\frac12$, we have that
$$|T|\,\geq\,2\pi\frac{(1-r^2)^2}{\left(1-\frac12r^2\right)^2}\,.$$
The number of disjoint balls contained in the $\Delta(Iy,r)$ is approximately the ratio between these areas, so it is greater then $c_5(r)d^{-4}$, as claimed.
\end{proof}

We can now prove that a result similar to Theorem \ref{tubo} using pseudohyperbolic balls for characterizing Carleson measures does not hold.

Notice that, due to Lemma \ref{lemma2.1} and Lemma \ref{d-lemma}, the Lebesgue measure of balls and their intersections with the $\mathbb C_I$ plane are equivalent to $d$ to some power. Hence a characterization of Carleson measure should be of the form

\emph{$\mu$ is Carleson if and only if $\mu(B(\alpha,r))\leq C d^\beta$ holds for all $r\in (0,1)$ for some constant $C$ depending on $r$ only, and for all $\alpha\in\mathbb B$, }

for some $\beta\in\mathbb R$.
But, actually:

\begin{theorem}Let $\mu$ be a finite, positive, Borel measure on $\mathbb B\subset \mathbb H$
\begin{enumerate}
\item If $\mu$ is Carleson, then $\mu(B(\alpha,r))\leq C d^4$ holds for all $r\in (0,1)$ for some constant $C$ depending on $r$ only, and for all $\alpha\in\mathbb B$.
\item If $\mu(B(\alpha,r))\leq C d^8$ holds for all $r\in (0,1)$ for some constant $C$ depending on $r$ only, and for all $\alpha\in\mathbb B$, then $\mu$ is Carleson.
\item The above are the two best estimates, i.e.\ one cannot fill the gap and get an \emph{if and only if} condition using the $\mu$-measure of pseudohyperbolic balls.
\end{enumerate}
\end{theorem}
\begin{proof}(1) Since $B(\alpha,r)\subset \Delta(\alpha,r)$, if $\mu$ is Carleson, the estimate on the $\mu$-volume of $\Delta(\alpha,r)$ gives the desired estimate on the $\mu$-volume of $B(\alpha,r)$.\vspace{0.3cm}\\

(2) By Lemma \ref{covering-tubes}, $\Delta(\alpha,r)$ can be covered with $C_5(r)d^{-4}$ balls $B(\alpha_j,4r)$, where $|\alpha_j|=|\alpha|$, hence $d(\alpha_j,\partial \mathbb B)=d$ for all $j$. If $\mu(B(\alpha,4r))\leq C d^8$, then $\mu(\Delta(\alpha,r))\leq C_5(r)C d^4$, hence $\mu$ is Carleson.\vspace{0.3cm}\\

(3-1) Estimate (1) is the best possible. Suppose indeed that $\mu$ is the Lebesgue measure on $\mathbb C_i$, $i\in\mathbb S$ fixed. Then $\mu(\Delta(\alpha,r))=|\Delta_i(\alpha,r)|$ if $\alpha\in \mathbb C_i$, and is less then that if $\alpha\not\in \mathbb C_i$. Hence $\mu$ is Carleson, and $\mu(B(\alpha,r)= |\Delta_i(\alpha,r)|\geq c_3(r)d^4$ if $\alpha\in\mathbb C_i$ (thanks to Lemma \ref{d-lemma}).\vspace{0.3cm}\\

(3-2) Estimate (2) is the best possible. Fix $r$. Let us fix a plane $\C_I$ and a sequence of points $\alpha_k=Iy_k\in \mathbb B_I$ ($k\in\mathbb N$) such that $y_k<y_{k+1}$ and the domains $\Delta(\alpha_k,2r)$ are pairwise disjoint. This can easily be done defining the points $Iy_k$ by induction on $k$. Let us define the measure $\mu$ on $\mathbb B$ such that
\begin{enumerate}
\item[(i)] ${\emph supp}\,\mu\ =\ \bigcup_{k\in\mathbb N}\,\Delta(\alpha_k,r)$;
\item[(ii)] $\mu_{|_{\Delta(\alpha_k,r)}}=\frac{d(\alpha_k,\partial \mathbb B)^{4-\varepsilon}}{\eta(\Delta(\alpha_k,r))}\eta$.
\end{enumerate}
By Theorem \ref{tubo}, for every $\varepsilon>0$, $\mu$ is not a Carleson measure.

By Lemma \ref{covering-tubes}, there exists a constant $c_5(r)$ such that for each $k\in\mathbb N$ there are at least $c_5(r)d(\alpha_k,\partial \mathbb B)^{-4}$ disjoint pseudohyperbolic balls of centers $\alpha_k^j=J_{j,k}y_k$, where $J_{j,k}\in\mathbb S$, and pseudohyperbolic radius $r$. For these balls

$$\sum_{j}\mu(B(\alpha_k^j,r))\ \leq\ \mu(\Delta(\alpha_k,r))\ =\ d(\alpha_k,\partial \mathbb B)^{4-\varepsilon}$$
hence
$$\mu(B(\alpha_k^j,r))\ \leq\  \frac1{c_5(r)}\,d(\alpha_k,\partial \mathbb B)^{8-\varepsilon}\,.$$

Let now $B(\alpha,r)$ be any pseudohyperbolic ball of pseudohyperbolic radius $r$. Then it intersect at most one of the domains $\Delta(\alpha_k,r)$, let us say $\Delta(\alpha_{\widetilde{k}},r)$. Arguing as before, this ball is one of a family of at least $c_5(r)d(\alpha,\partial \mathbb B)^{-4}$ pseudohyperbolic balls with disjoint intersection with $\Delta(\alpha_{\widetilde{k}},r)$. For these balls

$$\sum_{j}\mu(B_j)\ \leq\ \mu(\Delta(\alpha_{\widetilde{k}},r))\ =\ d(\alpha_{\widetilde{k}},\partial \mathbb B)^{4-\varepsilon}$$

hence

$$\mu(B(\alpha,r))\ \leq\  \frac1{c_5(r)}\,d(\alpha_{\widetilde{k}},\partial \mathbb B)^{8-\varepsilon}\,.$$

Since $\alpha_{\widetilde{k}}\in B(\alpha,2r)$, by Lemma \ref{lemma2.2},
$$d(\alpha_{\widetilde{k}},\partial \mathbb B)\ \leq\ \frac{C_1}{1-2r} d(\alpha,\partial\mathbb B)\,,$$
hence
$$\mu(B(\alpha,r))\ \leq\ C_6(r)  d(\alpha,\partial \mathbb B)^{8-\varepsilon}\,,$$

for every $\alpha\in\mathbb B$. This show that the $8^{\emph th}$ power in (2) is the best possible.

\end{proof}

\end{document}